\documentclass[a4paper,12pt]{amsart}

\usepackage{amsmath}
\usepackage{amssymb}
\usepackage{amsthm}
\usepackage[mathscr]{eucal}

\newcommand{\Ant}{\operatorname{Ant}}
\newcommand{\CAT}{\operatorname{CAT}}
\newcommand{\CBA}{\operatorname{CBA}}
\newcommand{\CBB}{\operatorname{CBB}}
\newcommand{\diam}{\operatorname{diam}}
\newcommand{\GCBA}{\operatorname{GCBA}}
\newcommand{\ulim}{\operatorname{\lim_{\omega}}}

\newcommand{\R}{\operatorname{\mathbb{R}}}

\newcommand{\N}{\operatorname{\mathbb{N}}}
\newcommand{\Sph}{\operatorname{\mathbb{S}}}
\newcommand{\GH}{\operatorname{\mathrm{GH}}}
\newcommand{\T}{\operatorname{\mathrm{T}}}

\numberwithin{equation}{section}

\theoremstyle{plain}
\newtheorem{thm}{Theorem}[section]
\newtheorem{lem}[thm]{Lemma}  
\newtheorem{prop}[thm]{Proposition}

\theoremstyle{definition}
\newtheorem{defn}{Definition}[section]

\theoremstyle{remark}
\newtheorem{rem}{Remark}[section]

\begin{document} 

\title
[Asymptotic geometric regularity of CAT(0) spaces]
{Asymptotic geometric regularity of \\ CAT(0) spaces}

\author
[Koichi Nagano]{Koichi Nagano}

\thanks{
Partially supported
by JSPS KAKENHI Grant Numbers 20K03603, 25K06996}

\address
[Koichi Nagano]
{\endgraf 
Department of Mathematics, University of Tsukuba
\endgraf
Tennodai 1-1-1, Tsukuba, Ibaraki, 305-8571, Japan}

\email{nagano@math.tsukuba.ac.jp}

\date{February 24, 2026}


\begin{abstract}
We prove that if an $n$-dimensional geodesically complete 
$\CAT(0)$ space has Tits boundary sufficiently close to 
the $(n-1)$-dimensional standard unit sphere,
then it is bi-Lipschiz 
homeomorphic to the $n$-dimensional Euclidean space.
As an application,
we conclude that
if an $(n-1)$-dimensional geodesically complete 
$\CAT(1)$ space is sufficiently close to 
the $(n-1)$-dimensional standard unit sphere,
then they are bi-Lipschiz 
homeomorphic to each other.
\end{abstract}

\maketitle


\section{Introduction}

In this paper,
from a viewpoint of asymptotic geometry,
we study a problem
when a $\CAT(0)$ space is bi-Lipschitz homeomorphic to
the $n$-dimensional Euclidean space $\R^n$.

As formulated by Gromov \cite{gromovh},
a $\CAT(0)$ space 
is defined as a complete geodesic metric space
that is globally non-positively curved in the sense of Alexandrov.
We recall that
a connected complete Riemannian manifold is $\CAT(0)$
if and only if it is a Hadamard manifold,
namely,
a simply connected complete Riemannian manifold 
of non-positive sectional curvature;
in this case,
it is diffeomorphic to $\R^n$.

Let $X$ be a $\CAT(0)$ space.
The ideal boundary $\partial_{\infty}X$ of $X$ is defined as 
the set of all asymptotic equivalence classes of rays in $X$.
The Tits boundary $\partial_{\mathrm{T}}X$ of $X$
is defined as
the ideal boundary $\partial_{\infty}X$ of $X$ equipped with
the Tits metric $d_{\mathrm{T}}$
(see Subsection 2.6).
We notice that $\partial_{\mathrm{T}}X$ is a $\CAT(1)$ space, 
whenever $X$ admits a ray.
If $X$ is a proper, geodesically complete $\CAT(0)$ space, 
and if the Tits boundary $\partial_{\mathrm{T}}X$ of $X$ is isometric to 
the $(n-1)$-dimensional standard unit sphere $\Sph^{n-1}$, 
then $X$ is isometric to $\R^n$. 
This rigidity follows from an observation of Leeb \cite[Proposition 2.1]{leeb}, obtained as a generalization of Schroeder's work 
in \cite[Appendix 4]{ballmann-gromov-schroeder} for Hadamard manifolds.
As a result of asymptotic geometric regularity,
the author \cite[Theorems 5.10 and 5.11]{nagano5} proved 
that 
if $X$ is a proper, geodesically complete $\CAT(0)$ space, 
and if $\partial_{\T}X$
is sufficiently close to $\Sph^{n-1}$
with respect to the Gromov--Hausdorff distance,
then $X$ is bi-Lipschitz homeomorphic to $\R^n$.

Our main purpose is
to generalize the results of asymptotic geometric regularity
mentioned above
without the properness assumption.

\subsection{Main theorems}

Throughout this paper,
we denote by $\dim$ 
the covering (topological) dimension of metric spaces.
We denote by $d_{\GH}$ 
the Gromov--Hausdorff distance between metric spaces.

As the main result of this paper,
we prove the following asymptotic geometric regularity theorem
for $\CAT(0)$ spaces:

\begin{thm}
\label{thm: agrcat0}
For every $\epsilon \in (0,\infty)$,
and for every $n \in \N$,
there exists $\delta \in (0,\infty)$ such that
if a separable, geodesically complete $\CAT(0)$ space $X$
of $\dim X \le n$
satisfies
$d_{\GH} \left( \partial_{\mathrm{T}}X, \Sph^{n-1} \right) < \delta$,
then $X$ is $(1+\epsilon)$-bi-Lipschitz homeomorphic to $\R^n$.
\end{thm}

In Theorem \ref{thm: agrcat0},
a desired bi-Lipschitz homeomorphism 
is given by a map on $X$ with Busemann function coordinates.
When we analyze such a regular map on $X$ 
with Busemann function coordinates,
we utilize the ideas of the theory 
of strainer maps on $\GCBA$ spaces 
with distance function coordinates
developed by 
Lytchak and the author \cite{lytchak-nagano1}.
In fact,
we need a generalization to 
the case of geodesically complete $\CBA$ spaces 
recently examined by the author \cite{nagano6}.

\begin{rem}
The author \cite{nagano5} studied asymptotic topological regularity
of $\CAT(0)$ spaces in terms of volume growth,
and proved that if a purely $n$-dimensional,
proper, geodesically complete $\CAT(0)$ space 
has small Euclidean volume growth,
then it is homeomorphic to $\R^n$.
\end{rem}

\begin{rem}
Gromov \cite{gromovq} asked the question whether 
there exists a Busemann convex geodesic metric space 
that is a topological $n$-manifold differ from $\R^n$.
We note that every $\CAT(0)$ space is Busemann convex.
For the case of $n \ge 5$,
Davis--Januszkiewicz \cite[Theorem 5b.1]{davis-januszkiewicz}
gave an affirmative answer,
in fact,
showed that for each $n \in \N$ with $n \ge 5$,
there exists a piecewise flat $\CAT(0)$ polyhedron
that is a topological $n$-manifold not homeomorphic to $\R^n$
(see also \cite{ancel-davis-guilbault}).
In the case of $n = 3$,
if a Busemann convex geodesic metric space 
is a topological $3$-manifold,
then it is homeomorphic to $\R^3$ 
(\cite{brown}, \cite{rolfsen}, \cite{thurston}).
In the case of $n = 4$,
Lytchak, Stadler, and the author \cite{lytchak-nagano-stadler}
proved that
every $\CAT(0)$ topological $4$-manifold
is homeomorphic to $\R^4$.
Fujioka--Gu \cite{fujioka-gu} have recently shown that
every Busemann convex topological $4$-manifold
is homeomorphic to $\R^4$.
\end{rem}

As an application of Theorem \ref{thm: agrcat0},
we conclude the following geometric regularity theorem
for $\CAT(1)$ spaces:

\begin{thm}
\label{thm: grcat1}
For every $\epsilon \in (0,\infty)$,
and for every $n \in \N$,
there exists $\delta \in (0,\infty)$ such that
if a separable, geodesically complete $\CAT(1)$ space $Z$
of $\dim Z \le n-1$ satisfies
$d_{\GH} \left( Z, \Sph^{n-1} \right) < \delta$,
then $Z$ is $(1+\epsilon)$-bi-Lipschitz homeomorphic to $\Sph^{n-1}$.
\end{thm}

To derive Theorem \ref{thm: grcat1} 
from Theorem \ref{thm: agrcat0},
we utilize the ideas of the studies 
of geometric regularity of $\CBB$ spaces
with curvature bounded below
developed by Burago--Gromov--Perelman 
\cite[Section 9]{burago-gromov-perelman}.
When actually implementing the ideas,
we need the studies of the author \cite{nagano6}
on almost spherical suspension structure of $\CAT(1)$ spaces.

\begin{rem}
Lytchak and the author \cite[Theorem 8.3]{lytchak-nagano2}
proved that 
if a purely $(n-1)$-dimensional,
compact, geodesically complete $\CAT(1)$ space 
has small volume,
then it is homeomorphic to $\Sph^{n-1}$.
The author \cite{nagano4} studied volume pinching problems 
for $\CAT(1)$ spaces.
\end{rem}

\subsection{Organization}

In Section 2,
we prepare basic concepts in the geometry 
of metric spaces with curvature bounded above,
especially $\CAT(0)$ spaces.
In Section 3,
we quote the studies of the author \cite{nagano6}
concerning an almost multi-fold spherical suspension structure
for $\CAT(1)$ spaces,
and show basic lemmas for the proofs of the main theorems.
In Section 4,
we formulate the notion of Busemann strainer maps on 
geodesically complete $\CAT(0)$ spaces,
and investigate their various regularity.
In Section 5,
we prove Theorems \ref{thm: agrcat0} and \ref{thm: grcat1}.

\subsection*{Acknowledgments}

The author would like to express his gratitude to
Alexander Lytchak for valuable discussions 
in private communications.
The author would like to thank
Takashi Shioya and Takao Yamaguchi
for their interests in this work.

\section{Preliminaries}

We refer the readers to 
\cite{alexander-kapovitch-petrunin-0},
\cite{alexander-kapovitch-petrunin}, 
\cite{ballmann}, \cite{bridson-haefliger}, \cite{burago-burago-ivanov}
for the basic facts
on the geometry of
metric spaces with an upper curvature bound.

\subsection{Metric spaces}

Let $r \in (0,\infty)$.
For a point $p$ in a metric space,
we denote by $U_r(p)$, $B_r(p)$, and $S_r(p)$
the open metric ball of radius $r$ around $p$,
the closed one, and the metric sphere, respectively.

For a metric space $X$ with metric $d_X$,
and for a positive number $\lambda$,
let $\lambda X$ denote the rescaled metric space
defined as $(X, \lambda d_X)$.

Let $X$ be a metric space with metric $d_X$.
Let $d_X \wedge \pi$ be the $\pi$-truncated metric on $X$
defined by $d_X \wedge \pi := \min \{ d_X, \pi \}$.
The 
\emph{Euclidean cone $C_0(X)$ over $X$}
is defined as the cone 
$[0,\infty) \times X / \{ 0 \} \times X$
over $X$
equipped with the Euclidean metric
$d_{C_0(X)}$ given by
\[
d_{C_0(X)} \left( [(t_1,x_1)], [(t_2,x_2)] \right)^2 := 
t_1^2 + t_2^2 - 2t_1t_2 \, \cos \left( d_X \wedge \pi \right) (x_1,x_2).
\]
For simplicity,
we write an element $[(t,x)]$ in $C_0(X)$ as $tx$,
and denote by $0$ the vertex of $C_0(X)$.

\subsection{Maps between metric spaces}

Let $c \in (0,\infty)$.
Let $X$ be a metric space with metric $d_X$,
and $Y$ a metric space with metric $d_Y$.
A map $\varphi \colon X \to Y$ is said to be
\emph{$c$-Lipschitz} 
if 
$d_Y \left( \varphi(x_1), \varphi(x_2) \right) \le c \, d_X(x_1,x_2)$
for all $x_1, x_2 \in X$.
A map $\varphi \colon X \to Y$ is said to be
\emph{$c$-bi-Lipschitz} 
if $f$ is $c$-Lipschitz,
and if $d_X(x_1,x_2) \le c \, d_Y \left( \varphi(x_1), \varphi(x_2) \right)$
for all $x_1, x_2 \in X$
(consequently, $c \in [1,\infty)$).
A $1$-bi-Lipschitz homeomorphism
is nothing but an isometry,
and a $1$-bi-Lipschitz embedding is 
an isometric embedding.
A map $\varphi \colon X \to Y$ is 
\emph{$c$-open}
if for any $r \in (0,\infty)$ and for any $x \in X$
such that $B_{cr}(x)$ is complete,
the ball $U_r \left( \varphi(x) \right)$ in $Y$ 
is contained in the image $\varphi \left( U_{cr}(x) \right)$
of the ball $U_{cr}(x)$ in $X$.
In the case where $X$ is complete,
if a map $\varphi \colon X \to Y$ is $c$-open,
then $\varphi$ is surjective.
Moreover,
in the case where $X$ is complete,
a $c$-Lipschitz map $\varphi \colon X \to Y$
is a $c$-bi-Lipschitz homeomorphism for some $c \in [1,\infty)$
if and only if $\varphi$ is an injective $c$-open map.

A map $\varphi \colon X \to Y$ is said to be a
\emph{$c$-approximation} between $X$ and $Y$
if $\bigcup_{y \in \varphi(X)} U_r(y)$ coincides with $Y$,
and if for all $x_1, x_2 \in X$ we have
\[
\left\vert
d_Y \left( \varphi(x_1), \varphi(x_2) \right) - d_X(x_1,x_2)
\right\vert
< c.
\]
If there exists a $c$-approximation $\varphi \colon X \to Y$,
then there exists a $2c$-approximation $\psi \colon Y \to X$ such that
for all $x \in X$ and $y \in Y$ we have
$d_X \left( (\psi \circ \varphi)(x), x \right) < 2c$ and
$d_Y \left( (\varphi \circ \psi)(y), y \right) < 2c$.

\subsection{Geodesic metric spaces}

Let $X$ be a metric space.
A \emph{geodesic $\gamma \colon I \to X$} 
means an isometric embedding from an interval $I$.
For a pair of points $p, q$ in $X$,
a \emph{geodesic $pq$ in $X$ from $p$ to $q$}
means the image of an isometric embedding
$\gamma \colon [a,b] \to X$
from a bounded closed interval $[a,b]$
with $\gamma(a) = p$ and $\gamma(b) = q$.
A geodesic $\gamma \colon I \to X$ is called 
a \emph{ray} 
if $I = [0,\infty)$,
and $\gamma$ is called a \emph{line} if $I = \R$.

For $r \in (0,\infty]$,
a metric space $X$ is said to be
\emph{$r$-geodesic}
if every pair of points in $X$ with distance smaller than $r$
can be joined by a geodesic in $X$.
A metric space is 
\emph{geodesic}
if it is $\infty$-geodesic.
A metric space is 
\emph{proper}
if every closed bounded subset is compact.
A geodesic metric space is proper
if and only if
it is complete and locally compact.

For $r \in (0,\infty]$,
a subset $C$ of a metric space is said to be 
\emph{$r$-convex}
if $C$ itself is $r$-geodesic as a metric subspace, 
and if every geodesic joining two points in $C$
is contained in $C$.
A subset $C$ of a metric space is 
\emph{convex}
if $C$ is $\infty$-convex.

\subsection{Gromov--Hausdorff topology}

As mentioned in Section 1,
we denote by $d_{\GH}$ the Gromov--Hausdorff distance 
between metric spaces.
If for $c \in (0,\infty)$ two metric spaces $X$ and $Y$
satisfy $d_{\GH}(X,Y) < c$,
then there exists a $2c$-approximation between $X$ and $Y$.
We say that a sequence $(X_k)$ of metric spaces 
converges to a metric space $X$ 
\emph{in the Gromov--Hausdorff topology}
if $\lim_{k \to \infty} d_{\GH}(X_k,X) = 0$.

A sequence $(X_k,p_k)$ of pointed geodesic metric spaces 
converges to a pointed metric space $(X,p)$
\emph{in the pointed Gromov--Hausdorff topology}
if for every $r \in (0,\infty)$
there exists a sequence $(\epsilon_k)$ in $(0,\infty)$
with $\lim_{k \to \infty}\epsilon_k = 0$
such that for each $k \in \N$
there exists an
$\epsilon_k$-approximation $\varphi_k \colon B_r(p) \to B_r(p_k)$
with $\varphi_k(p) = p_k$;
in this case,
we write $(X,p) = \lim_{k \to \infty} (X_k,p_k)$.
If a sequence $(X_k,p_k)$ of pointed proper geodesic metric spaces 
converges to a pointed metric space $(X,p)$
in the pointed Gromov--Hausdorff topology,
then $X$ is proper and geodesic.

A \emph{non-principal ultrafilter $\omega$ on $\N$}
is a finitely additive probability measure on $\N$ such that
$\omega(A) \in \{0,1\}$ for every subset $A$ of $\N$,
and $\omega(E) = 0$ if $E$ is a finite subset.
For a subset $A$ of $\N$ with $\omega(A) = 1$,
each $k \in A$ is said to be \emph{$\omega$-large}.
If $(x_k)$ is a sequence in a compact metric space $B$,
then there exists a unique point $x$ in $B$ such that
for every open neighborhood $U$ of $x$
we have 
$\omega \left( \left\{ \, k \in \N \mid x_k \in U \, \right\} \right) = 1$;
the unique point $x$ is called the \emph{$\omega$-limit of $(x_k)$},
and denoted by $\ulim x_k$.

Throughout this paper,
we fix a non-principal ultrafilter $\omega$ on $\N$.
Let $(X_k,p_k)$ be a sequence of pointed metric spaces
with metrics $d_{X_k}$.
We denote by $X_{\omega}^0$
the set of all sequences $(x_k)$
with $x_k \in X_k$ 
such that
$d_{X_k}(p_k,x_k)$ are uniformly bounded.
We denote by 
$d_{\omega}$ the function on 
$X_{\omega}^0 \times X_{\omega}^0$
defined by
$d_{\omega}((x_k),(y_k)) := \ulim d_{X_k}(x_k,y_k)$.
We define 
$\ulim (X_k,p_k)$ as the quotient
metric space $(X_{\omega}^0,d_{\omega}) / d_{\omega}=0$,
called the \emph{ultralimit of $(X_k,p_k)$ 
with respect to $\omega$}.
For $\xi \in \ulim X_k$,
if $\xi = [(x_k)]$,
then we write $x_k \to \xi$.

Let $(X_k,p_k)$ be a sequence of pointed metric spaces.
Then $\ulim X_k$ is complete.
If $\diam X_k$ are uniformly bounded,
then $\ulim X_k$ does not depend on the choices of
$p_k$.
If each $X_k$ is $r$-geodesic,
then so is $\ulim X_k$.
If each $X_k$ is geodesically complete,
then so is $\ulim X_k$.
If a sequence $(X_k)$ of compact metric spaces 
converges to some compact metric space $X$ 
in the Gromov--Hausdorff topology,
then $\ulim X_k$ is isometric to $X$.
If a sequence $(X_k,p_k)$ of pointed proper geodesic metric spaces 
converges to a pointed metric space $(X,p)$
in the pointed Gromov--Hausdorff topology,
then $\ulim X_k$ is isometric to $X$ and 
$p_k \to p$.

Let $(\varphi_k)$ be a sequence of pointed $c$-Lipschitz maps
from $(X_k,p_k)$ to $(Y_k,\varphi_k(p_k))$
between pointed metric spaces.
The \emph{ultralimit $\ulim \varphi_k$ of $(\varphi_k)$}
is defined as a map from 
from $\ulim (X_k,p_k)$ to $\ulim (Y_k,\varphi_k(p_k))$
uniquely determined by
$\left( \ulim \varphi_k \right) \left( [x_k] \right)
:= \left[ \left( \varphi_k(x_k) \right) \right]$.
In this case,
the ultralimit $\ulim \varphi_k$ is $c$-Lipschitz.
If in addition each $\varphi_k$ is $c$-open,
then $\ulim \varphi_k$ is $c$-open too.

\subsection{CAT$\boldsymbol{(\kappa)}$ spaces}

For $\kappa \in \R$,
we denote by $M_{\kappa}^n$ 
the simply connected, complete Riemannian $n$-manifold 
of constant curvature $\kappa$,
and denote by $D_{\kappa}$ the diameter of $M_{\kappa}^n$.
A complete metric space $X$ is said to be 
$\CAT(\kappa)$
if $X$ is $D_{\kappa}$-geodesic,
and if every geodesic triangle in $X$ with perimeter 
smaller than $2D_{\kappa}$
is not thicker than the comparison triangle 
with the same side lengths in $M_{\kappa}^2$.
In this paper,
all $\CAT(\kappa)$ spaces are assumed to be complete.
A metric space is said to be $\CBA$
if for some $\kappa \in \R$
it is locally $\CAT(\kappa)$.

Let $X$ be a $\CAT(\kappa)$ space.
Every pair of points in $X$ with distance smaller than $D_{\kappa}$
can be uniquely joined by a geodesic.
Let $p \in X$.
For every $r \in (0,D_{\kappa}/2]$,
the balls $U_r(p)$ and $B_r(p)$ are convex.
For every $r \in (0,D_{\kappa})$,
the balls $U_r(p)$ and $B_r(p)$ are contractible inside themselves.
For $x, y \in U_{D_{\kappa}}(p) - \{p\}$,
we denote by $\angle_p(x,y)$
the angle at $p$
between $px$ and $py$.
Put $\Sigma_p'X := \{ \, px \mid x \in U_{D_{\kappa}}(p) - \{p\} \, \}$.
The angle $\angle_p$ at $p$
is a pseudo-metric on $\Sigma_p'X$.
The 
\emph{space of directions $\Sigma_pX$ at $p$ in $X$}
is defined as the $\angle_p$-completion of
the quotient metric space $\Sigma_p'X / \angle_p = 0$.
For $x \in U_{D_{\kappa}}(p) - \{p\}$,
we denote by $x_p' \in \Sigma_pX$
the starting direction of $px$ at $p$.
The 
\emph{tangent space $T_pX$ at $p$ in $X$} 
is defined as $C_0(\Sigma_pX)$.
The space $\Sigma_pX$ is $\CAT(1)$,
and the space $T_pX$ is $\CAT(0)$.
In fact,
for a metric space $Z$,
the Euclidean cone $C_0(Z)$ 
is $\CAT(0)$ if and only if $Z$ is $\CAT(1)$.

\subsection{Ideal boundaries of \boldmath$\CAT(0)$ spaces}

Let $X$ be a metric space with metric $d_X$.
Two rays $\gamma_1, \gamma_2 \colon [0,\infty) \to X$
are said to be
\emph{asymptotic}
if $\sup_{t \in [0,\infty)} d_X(\gamma_1(t),\gamma_2(t))$
is finite.
The asymptotic relation gives an equivalence relation
on the set of all rays in $X$.
The 
\emph{ideal boundary $\partial_{\infty}X$ of $X$}
is defined as the set of all asymptotic equivalence classes of rays in $X$.
For a ray $\gamma$ in $X$,
we denote by $\gamma(\infty)$
the asymptotic equivalent class of $\gamma$ in $\partial_{\infty}X$.

Let $X$ be a $\CAT(0)$ space.
For every $p \in X$, and for every $\xi \in \partial_{\infty}X$,
there exists a unique ray $\gamma \colon [0,\infty) \to X$
with $\gamma(0) = p$ and $\gamma(\infty) = \xi$.
For $p \in X$ and $\xi \in \partial_{\infty}X$,
we denote by $\gamma_{p\xi}$
the unique ray emanating from $p$ to $\xi$,
by $p\xi$ the image of $\gamma_{p\xi}$,
and by $\xi_p' \in \Sigma_pX$ the starting direction of $p\xi$ at $p$.
For $p \in X$, and $\xi, \eta \in \partial_{\infty}X$,
we denote by $\angle_p(\xi,\eta)$
the angle at $p$
between $p\xi$ and $p\eta$.
The 
\emph{angle metric $\angle$ on $\partial_{\infty}X$}
is defined by
$\angle(\xi,\eta) := \sup_{p \in X} 
\angle_p(\xi,\eta)$.
The 
\emph{Tits metric $d_{\T}$ on $\partial_{\infty}X$}
is defined as the length metric on $\partial_{\infty}X$ 
induced from $\angle$.
Notice that $\angle = \min \{ d_{\T}, \pi \}$,
and $d_{\T}$ possibly takes the value $\infty$.
We denote by $\partial_{\T}X$
the ideal boundary $\partial_{\infty}X$
equipped with the Tits metric $d_{\T}$,
and call it the
\emph{Tits boundary of $X$}.
Then $\partial_{\T}X$ is a $\CAT(1)$ space.
The Euclidean cone $C_0(\partial_{\T}X)$ 
is isometric to the Euclidean cone $C_0(\partial_{\infty}X)$ over 
the ideal boundary $\partial_{\infty}X$ with the angle metric $\angle$.

\subsection{Geodesically complete CAT$\boldsymbol{(\kappa)}$ spaces}

A $\CAT(\kappa)$ space is said to be 
\emph{locally geodesically complete}
(or has \emph{geodesic extension property})
if every geodesic defined on a compact interval can be extended to
a local geodesic beyond endpoints.
A $\CAT(\kappa)$ space is 
\emph{geodesically complete}
if every geodesic can be extended to a local geodesic defined on $\R$.
Every locally geodesically complete $\CAT(\kappa)$ space
is geodesically complete.
The geodesical completeness for 
$\CAT(\kappa)$ spaces
is preserved under the ultralimit.

We refer the readers to \cite{lytchak-nagano1} for 
the basic properties of $\GCBA$ spaces, that is,
locally compact, separable, locally geodesically complete
metric spaces with an upper curvature bound.
Let $X$ be a proper, geodesically complete $\CAT(\kappa)$ space.
For every $p \in X$,
the space $\Sigma_pX$ is compact and geodesically complete,
and $T_pX$ is proper and geodesically complete.
In fact,
for a $\CAT(1)$ space $Z$,
the Euclidean cone $C_0(Z)$ is geodesically complete
if and only if $Z$ is geodesically complete
and not a singleton.

\subsection{Dimensions of CAT$\boldsymbol{(\kappa)}$ spaces}

Kleiner \cite{kleiner} formulated
the \emph{geometric dimension} $\dim_{\mathrm{G}}$ 
for $\CBA$ spaces
as
the smallest function from the class of $\CBA$ spaces
to $\N \cup \{\infty\}$
such that
(1) $\dim_{\mathrm{G}} X = 0$ if and only if $X$ is discrete;
(2) $\dim_{\mathrm{G}} X \ge 1 + \dim_{\mathrm{G}} \Sigma_pX$ 
for all $p \in X$.
In this case,
we have
\[
\dim_{\mathrm{G}} X = 1 + \sup_{p \in X} \dim_{\mathrm{G}} \Sigma_pX.
\]
We quote the theorem of Kleiner \cite[Theorem A]{kleiner}
in the following form:

\begin{thm}
\label{thm: dimgdimt}
\emph{(\cite{kleiner})} 
Let $X$ be a complete $\CBA$ space.
Then $\dim_{\mathrm{G}} X$ coincides the supremum
of covering dimensions $\dim C$ of compact subsets $C$ of $X$;
if in addition $X$ is separable,
then $\dim_{\mathrm{G}} X = \dim X$.
\end{thm}

Let $\ulim (X_k,p_k)$ be the ultraimit of a sequence of 
pointed $\CAT(\kappa)$ spaces.
Then we have the following (\cite[Lemma 11.1]{lytchak2}):
\begin{equation}
\dim_{\mathrm{G}} \left( \ulim X_k \right) 
\le \ulim \dim_{\mathrm{G}} X_k.
\label{eqn: lscdim}
\end{equation}

Let $X$ be a $\GCBA$ space.
Every relatively compact open subset of $X$
has finite covering dimension
(see \cite[Subsection 5.3]{lytchak-nagano1}).
The covering dimension $\dim X$ 
is equal to the Hausdorff dimension of $X$,
and equal to the supremum of $m$
such that $X$ has an open subset $U$ homeomorphic to $\R^m$
(\cite[Theorem 1.1]{lytchak-nagano1}).

\section{Almost spherical suspension structure}

\subsection{Multi-fold suspenders in CAT(1) spaces}

Let $Z$ be a $\CAT(1)$ space with metric $d_Z$.
We say that 
a point $p_0$ in $Z$ is a \emph{suspender} in $Z$
if there exists another point $q_0$ in $Z$
such that for every $z \in Z$ we have
\begin{equation}
d_Z(p_0,z) + d_Z(z,q_0) = \pi.
\label{eqn: sph}
\end{equation}
In this case,
we say that $p_0$ and $q_0$ are \emph{opposite} to each other.
By \eqref{eqn: sph},
if $p_0$ and $q_0$ are mutually opposite suspenders in $Z$,
then $d_Z(p_0,q_0) = \pi$.

We say that an $m$-tuple $(p_1,\dots,p_m)$ of points in $Z$ is
an \emph{$m$-suspender} in $Z$
if there exists another $m$-tuple $(q_1,\dots,q_m)$ of points in $Z$
satisfying the following properties:
\begin{enumerate}
\item
for all $i \in \{ 1, \dots, m \}$,
the points $p_i$ and $q_i$ 
are mutually opposite suspenders in $Z$;
\item
for all distinct $i, j \in \{1, \dots, m\}$,
we have
$d_Z(p_i,p_j) = \pi/2$, 
$d_Z(p_i,q_j) = \pi/2$, 
and $d_Z(q_i,q_j) = \pi/2$.
\end{enumerate}
In this case, we say that $(p_1,\dots,p_m)$ and 
$(q_1,\dots,q_m)$ are \emph{opposite} to each other.

For a point $z$ in $Z$,
we say that
a point $\bar{z}$ in $Z$ is an \emph{antipode of $z$}
if we have $d_Z(z,\bar{z}) \ge \pi$.
We denote by $\Ant(z)$
the set of all antipodes of $z$,
and call it the \emph{antipodal set of $z$}.
We say that
a subset $A$ of $Z$ is
\emph{symmetric}
if $A$ contains $\bigcup_{p \in A} \Ant(p)$.

We say that a symmetric subset $A$ of $Z$ is 
\emph{round}
if every point in $A$ is a suspender in $Z$.
Notice that if $Z$ admits a round subset,
then $\diam Z = \pi$.
A subset $A$ of $Z$ is round
if and only if
for each $p \in A$
there exists a unique point $q$ in $A$ such that
$p$ and $q$ are mutually opposite suspenders in $Z$.

For a subset $A$ of $Z$,
we denote by $A^{\perp}$
the set of all points $z$ in $Z$ with $d_Z(z,A) \ge \pi/2$,
and call it the \emph{polar set of $A$}.

From the existence of an $m$-suspenders,
we find an $m$-fold spherical suspension structure
for $\CAT(1)$ spaces
(\cite[Proposition 3.7]{nagano6}).

\begin{prop}\label{prop: m0susp}
Let $Z$ be a $\CAT(1)$ space.
Assume that there exist mutually opposite $m$-suspenders
$(p_1,\dots,p_m)$ and $(q_1,\dots,q_m)$ in $Z$.
Then there exists 
a round, closed $\pi$-convex subset $C_{m-1}$ of $Z$
containing $(p_1,\dots,p_m)$ and $(q_1,\dots,q_m)$
such that $C_{m-1}$ is isometric to $\Sph^{m-1}$.
Moreover,
$Z$ isometrically splits as a spherical join
$C_{m-1} \ast C_{m-1}^{\perp}$,
where $C_{m-1}^{\perp}$ is the polar set of $C_{m-1}$.
\end{prop}

\subsection{Relaxed multi-fold suspenders in CAT(1) spaces}

Let $Z$ be a $\CAT(1)$ space with metric $d_Z$.
We say that $Z$ is 
\emph{penetrable}
if there exists a dense subset $Z_0$ of $Z$
such that for each $p \in Z_0$
there exists $q \in Z_0$ with $d_Z(p,q) \ge \pi$.
If $Z$ is penetrable, then $\diam Z \ge \pi$.
If $Z$ has at least two path-connected components,
then it is penetrable.
The penetrability of $\CAT(1)$ spaces is closed 
under the ultralimit.
Notice that if $Z$ is geodesically complete,
then $Z$ is penetrable,
provided $Z$ contains at least two points.

Let $\delta \in (0,1)$.
We say that 
a point $p_0$ in $Z$ is a \emph{$(1,\delta)$-suspender} 
if there exists another point $q_0$ in $Z$
satisfying
\begin{equation}
\sup_{z \in Z} \left\{ \, d_Z(p_0,z) + d_Z(z,q_0) \, \right\} < \pi + \delta.
\label{eqn: dsph}
\end{equation}
In this case,
we say that $p_0$ and $q_0$ are \emph{opposite} to each other.

The notion of the $(1,\delta)$-suspenders
is a perturbed one of the rigid $1$-suspenders
for penetrable $\CAT(1)$ spaces
(\cite[Lemma 4.1]{nagano6}).

\begin{lem}\label{lem: o1dsusp}
Let $Z$ be a penetrable $\CAT(1)$ space with metric $d_Z$.
Let $p_0$ and $q_0$ be mutually opposite 
$(1,\delta)$-suspenders in $Z$.
Then we have
\begin{equation}
\inf_{z \in Z} \left\{ \, d_Z(p_0,z) + d_Z(z,q_0) \, \right\} > \pi - \delta.
\label{eqn: o1dsuspa}
\end{equation}
\end{lem}

We say that an $m$-tuple $(p_1,\dots,p_m)$ of points in $Z$ is
an \emph{$(m,\delta)$-suspender} 
if there exists another $m$-tuple $(q_1,\dots,q_m)$ of points in $Z$
satisfying the following:
\begin{enumerate}
\item
for all $i \in \{ 1, \dots, m \}$,
the points $p_i$ and $q_i$ 
are mutually opposite $(1,\delta)$-suspenders in $Z$;
\item
for all distinct $i, j \in \{ 1, \dots, m \}$,
we have
\begin{equation}
d_Z(p_i,p_j) < \frac{\pi}{2} + \delta,
\quad
d_Z(p_i,q_j) < \frac{\pi}{2} + \delta,
\quad
d_Z(q_i,q_j) < \frac{\pi}{2} + \delta.
\label{eqn: dortho}
\end{equation}

\end{enumerate}
In this case, we say that $(p_1,\dots,p_m)$ and 
$(q_1,\dots,q_m)$ are \emph{opposite} to each other.

From Lemma \ref{lem: o1dsusp} we can derive the following
(\cite[Lemma 3.3]{nagano6}):

\begin{lem}\label{lem: omdsusp}
Let $Z$ be a penetrable $\CAT(1)$ space with metric $d_Z$.
Let $(p_1,\dots,p_m)$ and $(q_1,\dots,q_m)$ be mutually opposite 
$(m,\delta)$-suspenders in $Z$.
Then the following hold:
\begin{enumerate}
\item
for all $i \in \{ 1, \dots, m \}$
we have
\begin{equation}
\inf_{z \in Z} \left\{ \, d_Z(p_i,z) + d_Z(z,q_i) \, \right\} > \pi - \delta;
\label{eqn: omdsuspa}
\end{equation}
\item
for all distinct $i, j \in \{ 1, \dots, m \}$, 
we have
\begin{equation}
d_Z(p_i,p_j) > \frac{\pi}{2} - 2\delta,
\quad
d_Z(p_i,q_j) > \frac{\pi}{2} - 2\delta,
\quad
d_Z(q_i,q_j) > \frac{\pi}{2} - 2\delta.
\label{eqn: omdsuspb}
\end{equation}
\end{enumerate}
\end{lem}

\subsection{Dimensions versus relaxed multi-fold suspenders}

We already know a relation between 
the existence of relaxed multi-fold suspenders and the dimensions
(\cite[Lemma 4.9]{nagano6}).

\begin{lem}\label{lem: capmdsusp}
For every $n \in \N$
there exists $\delta \in (0,1)$
such that
if a penetrable $\CAT(1)$ space $Z$ 
of $\dim_{\mathrm{G}} Z \le n-1$
admits an $(m,\delta)$-suspender,
then $m \le n$. 
\end{lem}

The following can be shown by an ultralimit argument
combined with Proposition \ref{prop: m0susp}
and Lemma \ref{lem: omdsusp}
(\cite[Lemma 4.10]{nagano6}).

\begin{lem}\label{lem: fullsusp}
For every $n \in \N$, and for every $c \in [1,\infty)$,
there exists $\delta \in (0,1)$
such that
if a penetrable $\CAT(1)$ space $Z$ 
of $\dim_{\mathrm{G}} Z \le n-1$
admits an $(n,\delta)$-suspender $(p_1,\dots,p_n)$,
then no point $z$ in $Z$
satisfies 
\[
\max_{i \in \{ 1, \dots, n \}} 
\left\vert d_Z \left( p_i, z \right) - \frac{\pi}{2} \right\vert < c\delta.
\]
\end{lem}

In the process of the proof of Theorem \ref{thm: grcat1},
we will use the following:

\begin{lem}\label{lem: aperp}
For every $\epsilon \in (0,1)$,
and for every $n \in \N$,
there exists a sufficiently small $\delta \in (0,1)$ 
satisfying the following property:
Let $Z$ be a penetrable $\CAT(1)$ space with metric $d_Z$
of $\dim_{\mathrm{G}}Z \le n-1$
admitting an $(n,\delta)$-suspender
$(p_1,\dots,p_n)$.
If $z_1, z_2 \in Z$ satisfy 
\[
\left\vert d_Z(z_1,z_2) - \frac{\pi}{2} \right\vert < \delta,
\]
then we have
\[
\left\vert
\sum_{i=1}^n \cos d_Z \left( p_i, z_1 \right)
\cos d_Z \left( p_i, z_2 \right)
\right\vert
< \epsilon.
\]
\end{lem}

\begin{proof}
Let $\epsilon \in (0,1)$ and $n \in \N$.
Suppose that
for some sequence $(\delta_k)$ in $(0,\infty)$ with $\delta_k \to 0$
there exists a sequence $(Z_k)$
of penetrable $\CAT(1)$ spaces with metrics $d_{Z_k}$
of $\dim_{\mathrm{G}}Z_k \le n-1$
such that
each $Z_k$ admits an $(n,\delta_k)$-suspender.
Suppose in addition that
for some $z_{1,k}, z_{2,k} \in Z_k$
with 
\[
\left\vert d_{Z_k} \left( z_{1,k} ,z_{2,k} \right) - \frac{\pi}{2} \right\vert 
< \delta_k
\]
we have
\[
\left\vert
\sum_{i=1}^n \cos d_{Z_k} \left( p_{i,k}, z_{1,k} \right)
\cos d_{Z_k} \left( p_{i,k}, z_{2,k} \right)
\right\vert
\ge \epsilon.
\]
Choose the ultralimit $(Z_{\omega},z_{\omega})$ 
of a sequence of pointed metric spaces $(Z_k,z_k)$.
For each $i \in \{ 1, \dots, n \}$,
take $p_{i,\omega} \in Z_{\omega}$ with $p_{i,k} \to p_{i,\omega}$.
For each $j \in \{ 1, 2 \}$,
take $z_{j,\omega} \in Z_{\omega}$ with $z_{j,k} \to z_{j,\omega}$.
From \eqref{eqn: lscdim}
we derive $\dim_{\mathrm{G}}Z_{\omega} \le n-1$.
By Lemma \ref{lem: omdsusp}, 
the $m$-tuple 
$\left( p_{1,\omega}, \dots, p_{n,\omega} \right)$
is an $n$-suspender in $Z_{\omega}$.
From Proposition \ref{prop: m0susp}
it follows that $Z_{\omega}$ is isometric to $\Sph^{n-1}$.
Hence we have
\[
\sum_{i=1}^n \cos d_{Z_{\omega}} \left( p_{i,\omega}, z_{1,\omega} \right)
\cos d_{Z_{\omega}} \left( p_{i,\omega}, z_{2,\omega} \right) = 0.
\]
This implies that for $\omega$-large $k$ we have
\[
\left\vert
\sum_{i=1}^n \cos d_{Z_k} \left( p_{i,k}, z_{1,k} \right)
\cos d_{Z_k} \left( p_{i,k}, z_{2,k} \right)
\right\vert
< \epsilon.
\]
This a contradiction,
which proves the lemma.
\end{proof}

\section{Asymptotic geometric regularity}

\subsection{Busemann functions on CAT(0) spaces}

Let $X$ be a $\CAT(0)$ space with metric $d_X$.
Let $\gamma \colon [0,\infty) \to X$ be a ray in $X$.
The \emph{Busemann function $b_{\gamma} \colon X \to \R$} 
along $\gamma$
is defined by
\[
b_{\gamma}(x) := 
\lim_{t \to \infty} \left( d_X \left( x, \gamma(t) \right) - t \right).
\]
The Busemann function $b_{\gamma}$ is $1$-Lipschitz and convex.
For $r \in \R$, 
we denote by $B_r(\gamma)$
the closed $(-r)$-horoball $b_{\gamma}^{-1} ((-\infty,-r])$.
Note that
for every $x \in X$,
and for every $r \in (0,\infty)$ with $b_{\gamma}(x) > -r$,
we have
\begin{equation}
d_X \left( x, B_r(\gamma) \right) = b_{\gamma}(x) + r,
\label{eqn: horoball}
\end{equation}
where $d_X \left( x, B_r(\gamma) \right)$ is the distance 
from $x$ to $B_r(\gamma)$.
Let $\xi := \gamma(\infty)$.
Take $x \in X$.
For every $y \in X - \{ x \}$,
we have the first variation formula 
\begin{equation}
(b_{\gamma} \circ \gamma_{xy})_+'(a) = - \cos \angle_x (\xi,y)
\label{eqn: 1vfb}
\end{equation}
for Busemann functions,
where
$\gamma_{xy} \colon [a,b] \to X$ is the geodesic 
from $x$ to $y$,
and $(b_{\gamma} \circ \gamma_{xy})_+'(a)$
is the right derivative of $b_{\gamma} \circ \gamma_{xy}$ at $a$
(see e.g., \cite[Lemma 3.3]{fujiwara-nagano-shioya}).
Let $\gamma_{x\xi} \colon [0,\infty) \to X$ be the ray 
from $x$ to $\xi$.
For all $s, t \in [0,\infty)$ with $s < t$,
we have
\begin{equation}
b_{\gamma} \left( \gamma_{x\xi}(s) \right) - 
b_{\gamma} \left( \gamma_{x\xi}(t) \right)
= s-t.
\label{eqn: lebelb}
\end{equation}

We know the following basic property
(see e.g., \cite[Lemma 5.1]{nagano5}):

\begin{lem}\label{lem: busemann-level}
Let $X$ be a $\CAT(0)$ space.
Let $\gamma \colon [0,\infty) \to X$ be a ray in $X$.
Let $\xi := \gamma(\infty)$.
If for distinct points $x, y$ in $X$ we have
$b_{\gamma}(x) = b_{\gamma}(y)$,
then
$\angle_x (\xi,y) \le \pi/2$ and
$\angle_y (\xi,x) \le \pi/2$.
\end{lem}

\subsection{Strainers at infinity}

Similarly to \cite[Subsection 5.2]{nagano5},
we now introduce the following:

\begin{defn}\label{defn: strinf}
Let $m \in \N$ and $\delta \in (0,1)$.
Let $X$ be a $\CAT(0)$ space.
We say that
an $m$-tuple $(\xi_1, \dots, \xi_m)$ of points in $\partial_{\infty}X$ 
is an
\emph{$(m,\delta)$-strainer at infinity}
if it is an $(m,\delta)$-suspender in $\partial_{\mathrm{T}}X$.
If an $m$-tuple $(\xi_1, \dots, \xi_m)$ of points in $\partial_{\infty}X$ 
is an $(m,\delta)$-strainer at infinity,
then there exists another 
$m$-tuple $(\eta_1, \dots, \eta_m)$ of points in $\partial_{\infty}X$
such that
$(\xi_1, \dots, \xi_m)$ and 
$(\eta_1, \dots, \eta_m)$ are mutually opposite to each other.
\end{defn}

By the definition of the angle metric $\angle$ on $\partial_{\infty}X$,
we have:

\begin{lem}\label{lem: sphstrinf}
Let $X$ be a $\CAT(0)$ space,
and let $(\xi_1, \dots, \xi_m)$ be an $(m,\delta)$-strainer at infinity
in $\partial_{\mathrm{T}}X$.
Then for every $x \in X$ the $m$-tuple
$((\xi_1)_x', \dots, (\xi_m)_x')$ of the directions at $x$
is an $(m,\delta)$-suspender in $\Sigma_xX$.
Moreover,
if $(\xi_1, \dots, \xi_m)$ and $(\eta_1, \dots, \eta_m)$ 
are mutually opposite $(m,\delta)$-strainers at infinity
in $\partial_{\mathrm{T}}X$,
then 
$\left( (\xi_1)_x', \dots, (\xi_m)_x' \right)$ and 
$\left( (\eta_1)_x', \dots, (\eta_m)_x' \right)$
are mutually opposite $(m,\delta)$-suspenders in $\Sigma_xX$.
\end{lem}

Similarly to \cite[Lemma 5.3]{nagano6},
we verify that
the existence of strainers at infinity guarantees the 
existence of almost flat ideal triangles:

\begin{lem}\label{lem: almflattri}
Let $X$ be a $\CAT(0)$ space.
Assume that
a point $\xi$ in $\partial_{\mathrm{T}}X$ be a 
$(1,\delta)$-strainer at infinity. 
Then for every pair of distinct points $x, y$ in $X$
the following hold:
\begin{enumerate}
\item
$\pi - 2\delta < \angle_x (\xi,y) + \angle_y (\xi,x) \le \pi$;
\item
if $b_{\gamma}(x) = b_{\gamma}(y)$ for a ray 
$\gamma$ in $X$ with $\xi = \gamma(\infty)$,
then 
\[
\frac{\pi}{2} - 2\delta < \angle_x (\xi,y) \le \frac{\pi}{2}.
\]
\end{enumerate}
\end{lem}

\begin{proof}
Since $X$ is $\CAT(0)$,
we know $\angle_x (\xi,y) + \angle_y (\xi,x) \le \pi$
(see e.g., \cite[Proposition II.9.3]{bridson-haefliger}).
By Lemma \ref{lem: busemann-level},
if $b_{\gamma}(x) = b_{\gamma}(y)$ for a ray 
$\gamma$ in $X$ with $\xi = \gamma(\infty)$,
then $\angle_x (\xi,y) \le \pi/2$ and
$\angle_y (\xi,x) \le \pi/2$.

Take a point $\eta$ in $\partial_{\mathrm{T}}X$ 
for which $\xi$ and $\eta$
are opposite $(1,\delta)$-strainers at infinity.
Similarly to the case of $\xi$,
we know $\angle_x (\eta,y) + \angle_y (\eta,x) \le \pi$.
By Lemmas \ref{lem: o1dsusp} and \ref{lem: sphstrinf},
we have
\begin{align*}
\angle_x (\xi,y) &+ \angle_x (\eta,y) \ge \angle_x (\xi,\eta) > \pi - \delta, \\
\angle_y (\xi,x) &+ \angle_y (\eta,x) \ge \angle_y (\xi,\eta) > \pi - \delta.
\end{align*}
Therefore $\angle_x (\xi,y) + \angle_y (\xi,x) > \pi - 2\delta$.
Moreover,
if $b_{\gamma}(x) = b_{\gamma}(y)$,
then we obtain $\angle_x (\xi,y) > \pi/2 - 2\delta$.
This proves the lemma.
\end{proof}

\subsection{Strainer maps at infinity}

Let $X$ be a $\CAT(0)$ space with metric $d_X$.
For an $m$-tuple $(\gamma_1, \dots, \gamma_m)$
of rays in $X$,
define a map
$\varphi \colon X \to \R^m$
on $X$
by
$\varphi = \left( b_{\gamma_1}, \dots, b_{\gamma_m} \right)$,
where 
$\left( b_{\gamma_1}, \dots, b_{\gamma_m} \right)$
is the $m$-tuple of the Busemann functions.
For $i \in \{ 1, \dots, m \}$,
let $\xi_i := \gamma_i(\infty)$.
By the first variation formula \eqref{eqn: 1vfb},
the map $\varphi$ is differentiable at all $x \in X$
with differential $D_x\varphi \colon T_xX \to \R^m$ determined by
\[
\left( D_x\varphi \right) \left( rv \right)
= -r \left( \cos \angle_x \left( \left( \xi_1 \right)_x', v \right), \dots, 
\cos \angle_x \left( \left( \xi_m \right)_x', v \right) \right).
\]

\begin{defn}\label{defn: strm}
Let $m \in \N$ and $\delta \in (0,1)$.
Let $X$ be a geodesically complete $\CAT(0)$ space.
We say that a map
$\varphi \colon X \to \R^m$ 
is a \emph{Busemann $(m,\delta)$-strainer map} on $X$
if there exists an $m$-tuple 
$\left( \gamma_1, \dots, \gamma_m \right)$
of rays in $X$ 
with $\varphi = \left( b_{\gamma_1}, \dots, b_{\gamma_m} \right)$
such that
the $m$-tuple
$\left( \gamma_1(\infty), \dots, \gamma_m(\infty) \right)$
is an $(m,\delta)$-strainer at infinity
in $\partial_{\mathrm{T}}X$.
\end{defn}

For $u \in \R^m$
we denote by $\Vert u \Vert$
the Euclidean norm of $u$ on $\R^m$.

From Lemma \ref{lem: almflattri}
we derive the following first variation inequalities 
for Busemann strainer maps
(cf.~\cite[Proposition 8.5]{lytchak-nagano1}).

\begin{lem}\label{lem: 1stvfbstr0}
Let 
$\varphi \colon X \to \R^m$
be a Busemann $(m,\delta)$-strainer map on 
a geodesically complete $\CAT(0)$ space $X$
with $\varphi = \left( b_{\gamma_1}, \dots, b_{\gamma_m} \right)$.
Take a geodesic $\gamma \colon [a,b] \to X$ 
joining distinct two points in $X$.
Then for each $i \in \{ 1, \dots, m \}$ we have
\[
\left\vert
\frac{\left( b_{\gamma_i} \circ \gamma \right)(b) - 
\left( b_{\gamma_i} \circ \gamma \right)(a)}
{d_X \left( \gamma(b), \gamma(a) \right)}
- \left( b_{\gamma_i} \circ \gamma \right)^+(a)
\right\vert
< 2\delta.
\] 
In particular, we have
\[
\left\vert
\frac{\left\Vert \left( \varphi \circ \gamma \right)(b) - 
\left( \varphi \circ \gamma \right)(a) \right\Vert}
{d_X \left( \gamma(b), \gamma(a) \right)}
- \left\Vert \left( \varphi \circ \gamma \right)^+(a) \right\Vert
\right\vert
< 2 \sqrt{m} \delta,
\]
where 
$\left( \varphi \circ \gamma \right)^+(a) \in \R^m$
is defined by
\[
\left( \varphi \circ \gamma \right)^+(a) := 
\left( \left( b_{\gamma_1} \circ \gamma \right)^+(a), \dots, 
\left( b_{\gamma_m} \circ \gamma \right)^+(a) \right).
\]
If in addition for $j \in \{ 1, \dots, m \}$ we have
$\left( b_{\gamma_j} \circ \gamma \right)(a) = 
\left( b_{\gamma_j} \circ \gamma \right)(b)$,
then we have
\[
\left\Vert \left( b_{\gamma_j} \circ \gamma \right)^+(a) \right\Vert 
< 2\delta.
\]
Moreover,
if 
$\left( \varphi \circ \gamma \right)(a) = 
\left( \varphi \circ \gamma \right)(b)$,
then 
\[
\left\Vert \left( \varphi \circ \gamma \right)^+(a) \right\Vert 
< 2 \sqrt{m} \delta.
\]
\end{lem}

Lemma \ref{lem: 1stvfbstr0} 
leads to the following geometric property:

\begin{prop}\label{prop: 1stvfbstr}
Let 
$\varphi \colon X \to \R^m$
be a Busemann $(m,\delta)$-strainer map on 
a geodesically complete $\CAT(0)$ space $X$.
Let $\gamma \colon [a,b] \to X$ be a geodesic
joining distinct two points in $X$.
Then for all $s, t \in [a,b)$ we have
\[
\left\Vert \left( \varphi \circ \gamma \right)^+(s) - 
\left( \varphi \circ \gamma \right)^+(t) \right\Vert
\le 4 \sqrt{m} \delta.
\]
If for $s_1, s_2 \in [a,b)$ with $s_1 \neq s_2$
we have
$\left( \varphi \circ \gamma \right)(s_1) = 
\left( \varphi \circ \gamma \right)(s_2)$,
then for all $t \in [a,b)$ we have
\[
\left\Vert \left( \varphi \circ \gamma \right)^+(t) \right\Vert
\le 6 \sqrt{m} \delta.
\]
\end{prop}

The proof of Proposition \ref{prop: 1stvfbstr}
is left to the readers.

\subsection{Pseudo-strainer maps at infinity}

From \eqref{eqn: lebelb} it follows that
every Busemann function on a geodesically complete $\CAT(0)$ space
is $1$-open.

We are going to show that 
every Busemann $(m,\delta)$-strainer map is $c$-open
for some constant $c$ depending only on $m$ and $\delta$.
To do this,
we consider more general maps with Busemann function coordinates.

\begin{defn}\label{defn: bsstrm}
Let $m \in \N$ and $\delta \in (0,1)$.
Let $X$ be a geodesically complete $\CAT(0)$ space.
We say that a map $\varphi \colon X \to \R^m$ is a
\emph{Busemann $(m,\delta)$-pseudo-strainer map} on $X$
if there exists an $m$-tuple 
$\left( \gamma_1, \dots, \gamma_m \right)$
of rays in $X$ 
with $\varphi = \left( b_{\gamma_1}, \dots, b_{\gamma_m} \right)$
such that
for all $i, j \in \{ 1, \dots, m \}$ with $i \neq j$,
we have
\begin{equation}
\left\vert \angle_x \left( \gamma_i(\infty), \gamma_j(\infty) \right) -
\frac{\pi}{2} \right\vert
< 2\delta.
\label{eqn: bsstrma}
\end{equation}
\end{defn}

From a similar argument on an openness in 
\cite[Lemma 6.3]{nagano6},
we derive the following
(cf.~\cite[Lemma 8.1]{lytchak-nagano1} in the $\GCBA$ setting):

\begin{lem}\label{lem: wwopen}
For $m \in \N$, 
let $\delta \in (0,1)$ satisfy $2(m-1)\delta < 1$.
Let $\varphi \colon X \to \R^m$
be a Busemann $(m,\delta)$-pseudo-strainer map on 
a geodesically complete $\CAT(0)$ space $X$.
Then $\varphi$
is $1/(1-2(m-1)\delta)\sqrt{m}$-open.
In particular,
it is $2/\sqrt{m}$-open,
whenever $4(m-1)\delta < 1$.
\end{lem}

\begin{proof}
For $m \in \N$,
let $\delta \in (0,1)$ be sufficiently small.
Let $\varphi \colon X \to \R^m$
be a Busemann $(m,\delta)$-pseudo-strainer map
with $\varphi = \left( b_{\gamma_1}, \dots, b_{\gamma_m} \right)$.
For $i \in \{ 1, \dots, m \}$,
let $\xi_i := \gamma_i(\infty)$.
For a fixed point $x_0$ in $X$,
choose $r \in (0,\infty)$.
Take an element $u_0 = (u_0^1, \dots, u_0^m)$
in $\R^m$ with $\Vert u_0 \Vert < r/2$.
The goal is to find a point $y_{\ast}$ in $X$ 
with $\varphi(y_{\ast}) = \varphi(x_0) + u_0$
such that
$(1-2(m-1)\delta) d_X(x_0,y_{\ast}) \le \Vert u_0 \Vert_{\ell^1}$,
where
$d_X$ is the metric of $X$,
and
$\Vert u_0 \Vert_{\ell^1}$
is the $\ell^1$-norm of $u_0$. 
To do this,
we are going to construct $y_{\ast}$ as a limit point of a sequence
$y_1, y_2, \dots$.

First we find $y_1 \in X$ with 
$d_X(x_0,y_1) \le \Vert u_0 \Vert_{\ell^1}$
such that
\begin{equation}
\Vert \varphi(y_1) - (\varphi(x_0) + u_0) \Vert_{\ell^1}
\le 2(m-1) \delta \Vert u_0 \Vert_{\ell^1}.
\label{eqn: wwopen1}
\end{equation}
Let $(e_1,\dots,e_m)$ 
denote the standard orthonormal basis of $\R^m$.
For $u_0^1 \in \R$,
by the property \eqref{eqn: lebelb}
we can choose $x_1 \in X$ 
with $x_1 \in x_0\xi_1$ 
$\vert b_{\gamma_1}(x_1) - b_{\gamma_1}(x_0) \vert = \vert u_0^1 \vert$.
Notice that if $u_0^1 = 0$, then $x_1 = x_0$.
From our assumption \eqref{eqn: bsstrma}
and the first variation formula \eqref{eqn: 1vfb}
for Busemann functions,
we derive
$\vert b_{\gamma_i}(x_1) - b_{\gamma_i}(x_0) \vert < 2\delta \vert u_0^1 \vert$
whenever $i \neq 1$.
This implies
\[
\Vert \varphi(x_1) - (\varphi(x_0) + u_0^1 e_1) \Vert_{\ell^1}
\le 2(m-1) \delta \vert u_0^1 \vert.
\]
Successively,
for $u_0^m \in \R$,
we can choose $x_m \in X$ 
with $x_m \in x_{m-1}\xi_m$ 
and 
$\vert d_{p_m}(x_m) - d_{p_m}(x_{m-1}) \vert = \vert u_0^m \vert$;
in particular,
\[
\Vert
\varphi(x_m) - (\varphi(x_{m-1}) + u_0^m e_m ) \Vert_{\ell^1}
\le 2(m-1) \delta \vert u_0^m \vert.
\]
Put $y_1 := x_m$.
Then
we have $d_X(x_0,y_1) \le \Vert u_0 \Vert_{\ell^1}$
and \eqref{eqn: wwopen1}.

Inductively,
for $y_k \in X$ and $u_k := \varphi(x_0) + u_0 - \varphi(y_k)$,
we can construct 
$y_{k+1} \in X$ with $d_X(y_k,y_{k+1}) \le \Vert u_k \Vert_{\ell^1}$
such that
\[
\Vert \varphi(y_{k+1}) - (\varphi(y_k) + u_k) \Vert_{\ell^1}
\le 2(m-1) \delta \Vert u_k \Vert_{\ell^1}.
\]
For each $k \in \N$, we have
\begin{equation}
\Vert u_{k+1} \Vert_{\ell^1}
\le 2(m-1) \delta \Vert u_k \Vert_{\ell^1}
\le (2(m-1) \delta)^{k+1} \Vert u_0 \Vert_{\ell^1}
\label{eqn: wwopen2}
\end{equation}
\begin{equation}
d_X(x_0,y_{k+1}) \le \sum_{i=0}^kd_X(y_i,y_{i+1})
\le \frac{1-(2(m-1)\delta)^{k+1}}{1-2(m-1)\delta} 
\Vert u_0 \Vert_{\ell^1},
\label{eqn: wwopen3}
\end{equation}
where we set $y_0 := x_0$.
Hence $(y_k)$ is a Cauchy sequence in $B_r(x_0)$.
Let $y_{\ast}$ be the limit of $(y_k)$ in $B_r(x_0)$.
Then \eqref{eqn: wwopen2} implies 
$\varphi(y_{\ast}) = \varphi(x_0) + u_0$,
and \eqref{eqn: wwopen3}
does
$(1-2(m-1)\delta) d_X(x_0,y_{\ast}) \le \Vert u_0 \Vert_{\ell^1}$.
\end{proof}

As a corollary of Lemma \ref{lem: wwopen},
we see:

\begin{lem}\label{lem: wopen}
For $m \in \N$, 
let $\delta \in (0,1)$ satisfy $2(m-1)\delta < 1$.
Then every Busemann $(m,\delta)$-strainer map 
on a geodesically complete $\CAT(0)$ space
is $1/(1-2(m-1)\delta)\sqrt{m}$-open.
In particular,
it is $2/\sqrt{m}$-open,
whenever $4(m-1)\delta < 1$.
\end{lem}

\begin{proof}
By Lemma \ref{lem: omdsusp},
every Busemann $(m,\delta)$-strainer map is
a Busemann $(m,\delta)$-pseudo-strainer map.
From Lemma \ref{lem: wwopen},
we conclude the lemma.
\end{proof}

\subsection{Regularity of strainer maps at infinity}

We prove that every Busemann strainer map
is almost a submetry
(cf.~\cite[Subsection 8.3]{lytchak-nagano1}).

\begin{prop}
\label{prop: almsubm}
For every $\epsilon \in (0,1)$,
and for every $m \in \N$,
there exists $\delta \in (0,1)$ satisfying the following:
Let $X$ be a geodesically complete $\CAT(0)$ space.
Let $\varphi \colon X \to \R^m$ be 
a Busemann $(m,\delta)$-strainer map on $X$.
Then $\varphi$ is $(1+\epsilon)$-Lipschitz and $(1+\epsilon)$-open.
In particular,
for every $x \in X$
the directional derivative $D_x\varphi \colon T_xX \to \R^m$
of $\varphi$ at $x$
is $(1+\epsilon)$-Lipschitz and $(1+\epsilon)$-open.
\end{prop}

For the proof of Proposition \ref{prop: almsubm},
we prepare the following lemmas.

First we examine a rigid case.

\begin{lem}\label{lem: busemann}
Let $Z$ be a $\CAT(1)$ space.
Let $(\xi_1,\dots,\xi_m)$ be an $m$-suspender in $Z$.
For $i \in \{ 1, \dots, m \}$,
let $\gamma_i \colon [0,\infty) \to C_0(Z)$ be the ray in 
the Euclidean cone $C_0(Z)$ over $Z$
from the vertex $0$ toward $\xi_i$ 
defined by $\gamma_i(t) := t\xi_i$.
Then the map $\varphi \colon C_0(Z) \to \R^m$
defined by $\varphi := \left( b_{\gamma_1}, \dots, b_{\gamma_m} \right)$
is a homogenous submetry.
\end{lem}

\begin{proof}
By Proposition \ref{prop: m0susp},
there exists a round, closed $\pi$-convex subset $C_{m-1}$
of $Z$
containing $(\xi_1,\dots,\xi_m)$
such that $Z$ isometrically splits as 
$C_{m-1} \ast C_{m-1}^{\perp}$,
where
$C_{m-1}$ is isometric to $\Sph^{m-1}$,
and 
$C_{m-1}^{\perp}$ is the polar set of $C_{m-1}$.
The Euclidean cone $C_0(Z)$ isometrically splits as 
$\R^m \times C_0 \left( Y_{m-1}^{\perp} \right)$.
Observe that for each $s\xi \in C_0(Z)$, we have
\[
\varphi(s\xi) = 
-s 
\left( \cos d_Z(\xi_1, \xi), \dots, \cos d_Z(\xi_m, \xi) \right),
\]
where $d_Z$ is the metric on $Z$.
Hence the map
$-\varphi \colon C_0(Z) \to \R^m$
defined by $(-\varphi)(s\xi) := -\varphi(s\xi)$
coincides with 
the projection onto the $m$-dimensional Euclidean flat factor 
$\R^m$ in $C_0(Z)$.
Therefore $-\varphi$ is a submetry,
and hence so is $\varphi$.
The homogeneity follows from the definition of $\varphi$. 
\end{proof}

We next show the following:

\begin{lem}\label{lem: almuspd} 
For every $\epsilon \in (0,1)$, and for every $m \in \N$,
there exists $\delta \in (0,1)$ satisfying the following:
Let $X$ be a geodesically complete $\CAT(0)$ space.
Let $\varphi \colon X \to \R^m$
be a Busemann $(m,\delta)$-strainer map on $X$.
Then for every $x \in X$, and for every $\zeta_x \in \Sigma_xX$,
we have
$\Vert \left( D_x\varphi \right) \left( \zeta_x \right) \Vert < 1+\epsilon$.
\end{lem}

\begin{proof}
Suppose that for some sequence $(\delta_k)$ in $(0,\infty)$ 
with $\delta_k \to 0$,
there exists a sequence of $(m,\delta_k)$-strainer maps
$\varphi_k \colon X_k \to \R^m$
on geodesically complete $\CAT(0)$ spaces $X_k$
such that for some $\zeta_{x_k} \in \Sigma_{x_k}X_k$ 
we have
\begin{equation}
\Vert \left( D_{x_k}\varphi_k \right) \left( \zeta_{x_k} \right) \Vert 
\ge 1+\epsilon.
\label{eqn: almuspd1}
\end{equation}
For each $k \in \N$,
take the $m$-tuple 
$\left( \gamma_{1,k}, \dots, \gamma_{m,k} \right)$
of the rays in $X_k$
with
$\varphi_k = \left( b_{\gamma_{1,k}}, \dots, b_{\gamma_{m,k}} \right)$,
and for each $i \in \{ 1, \dots, m\}$
put $\xi_{i,k} := \gamma_{i,k}(\infty)$.
From Lemma \ref{lem: sphstrinf} it follows that
the $m$-tuple $\left( (\xi_{1,k})_{x_k}', \dots, (\xi_{m,k})_{x_k}' \right)$
is an $(m,\delta_k)$-suspender in $\Sigma_{x_k}X_k$.

Let 
$\left( \ulim T_{x_k}X_k, 0_{\omega} \right)$
be the ultralimit of the sequence $\left( T_{x_k}X_k, 0_{x_k} \right)$
of the pointed metric spaces,
and 
$\left( \ulim \Sigma_{x_k}X_k, \zeta_{\omega} \right)$
the ultralimit of the sequence $\left( \Sigma_{x_k}X_k, \zeta_{x_k} \right)$.
Since each $T_{x_k}X_k$ is defined as
$C_0 \left( \Sigma_{x_k}X_k \right)$,
the ultralimit $\left( \ulim T_{x_k}X_k, 0_{\omega} \right)$
is isometric to $\left( C_0 \left( \ulim \Sigma_{x_k}X_k \right), 0 \right)$.
From our present assumption and Lemma \ref{lem: omdsusp},
we see that
$\ulim \Sigma_{x_k}X_k$ admits an $m$-fold suspender 
$\left( \xi_{1,\omega}, \dots, \xi_{m,\omega} \right)$.

Take the ultralimit $\ulim D_{x_k}\varphi_k$
of the sequence of the maps $D_{x_k}\varphi_k$
defined on $\ulim T_{x_k}X_k$.
Then $\ulim D_{x_k}\varphi_k$ coincides with 
$\left( b_{\gamma_1}, \dots, b_{\gamma_m} \right)$, 
where $b_{\gamma_i}$ is the Busemann function on 
$C_0 \left( \ulim \Sigma_{x_k}X_k \right)$ 
along the ray $\gamma_i$ emanating from $0$ toward $\xi_{i,\omega}$.
By Lemma \ref{lem: busemann},
the map $\ulim D_{x_k}\varphi_k$ is a submetry.
For the point
$\zeta_{\omega}$ in 
$\ulim \Sigma_{x_k}X_k$ contained in $\ulim T_{x_k}X_k$
with $\zeta_{x_k} \to \zeta_{\omega}$.
Since $\ulim D_{x_k}\varphi_k$ is a submetry,
we have
\[
\left\Vert \left( \ulim D_{x_k}\varphi_k \right) \left( \zeta_{\omega} \right) \right\Vert \le 1.
\]
Hence $\Vert \left( D_{x_k}\varphi_k \right) \left( \zeta_{x_k} \right) \Vert 
< 1+\epsilon$
for $\omega$-large $k$.
This contradicts \eqref{eqn: almuspd1}.
\end{proof}

Lemma \ref{lem: almuspd} leads to the following:

\begin{lem}\label{lem: str1elip}
For every $\epsilon \in (0,1)$, and for every $m \in \N$,
there exists $\delta \in (0,1)$ satisfying the following:
Let $X$ be a geodesically complete $\CAT(0)$ space.
Let $\varphi \colon X \to \R^m$
be a Busemann $(m,\delta)$-strainer map on $X$.
Then $\varphi$ is $(1+\epsilon)$-Lipschitz.
\end{lem}

\begin{proof}
Let $\delta \in (0,1)$ be small enough.
For each pair of distinct points $x, y$ in $X$,
let $\gamma \colon [a,b] \to X$ be 
a (unique) minimizing geodesic from $x$ to $y$.
For the Lipschitz curve $\varphi \circ \gamma$ in $\R^m$, 
and for each $t \in [a,b)$,
we have 
\[
\left\Vert \left( \varphi \circ \gamma \right)^+(t) \Vert 
= \Vert \left( D_{\gamma(t)}\varphi \right) \left( y_{\gamma(t)}' \right) \right\Vert.
\]
From Lemma \ref{lem: almuspd},
we derive
\[
\left\Vert \varphi(x) - \varphi(y) \right\Vert \le 
\int_a^b
\left\Vert \left( \varphi \circ \gamma \right)^+(t) \right\Vert dt
< \left( 1+\epsilon \right) d_X(x,y),
\]
where $d_X$ is the metric on $X$.
Hence $\varphi$ is $(1+\epsilon)$-Lipschitz.
\end{proof}

Next we prove the following:

\begin{lem}\label{lem: str1eopen}
For every $\epsilon \in (0,1)$, and for every $m \in \N$,
there exists $\delta \in (0,1)$ satisfying the following:
Let $X$ be a geodesically complete $\CAT(0)$ space.
Let $\varphi \colon X \to \R^m$
be a Busemann $(m,\delta)$-strainer map on $X$.
Then for every $u \in \Sph^{m-1}$,
and for every $x \in X$,
we find $\eta_x \in T_xX$ with 
$\left( D_x\varphi \right) \left( \eta_x \right) = u$
such that
$\left\vert \eta_x \right\vert < 1+\epsilon$.
\end{lem}

\begin{proof}
Let $\delta \in (0,1)$ be sufficiently small.
As seen in Lemma \ref{lem: wopen},
for each $x \in X$ the directional derivative 
$D_x\varphi \colon T_xX \to \R^m$ of $\varphi$ at $x$
is $c$-open for some constant $c \in (1,\infty)$
depending only on $m$ and $\delta$;
in particular, 
for every $u \in \Sph^{m-1}$,
we find $\eta_x \in T_xX$ with 
$\left( D_x\varphi \right) \left( \eta_x \right) = u$
such that $\left\vert \eta_x \right\vert < c$.

To prove the lemma,
we suppose that for some sequence $(\delta_k)$ in $(0,\infty)$
with $\delta_k \to 0$,
there exists a sequence
of Busemann $(m,\delta_k)$-strainer maps
$\varphi_k \colon X_k \to \R^m$ on geodesically complete
$\CAT(0)$ spaces,
and a sequence $(x_k)$ with $x_k \in X_k$
such that
for some $u_k \in \Sph^{m-1}$, and
for all $\eta_{x_k} \in T_{x_k}X_k$ with 
$\left( D_{x_k}\varphi_k \right) \left( \eta_{x_k} \right) = u_k$
we have
\begin{equation}
\left\vert \eta_{x_k} \right\vert \ge 1+\epsilon.
\label{eqn: str1eopen1}
\end{equation}

Similarly to the proof of Lemma \ref{lem: almuspd},
let 
$\left( \ulim T_{x_k}X_k, 0_{\omega} \right)$
be the ultralimit of the sequence $\left( T_{x_k}X_k, 0_{x_k} \right)$
of the pointed metric spaces,
and 
$\left( \ulim \Sigma_{x_k}X_k, \zeta_{\omega} \right)$
the ultralimit of a sequence $\left( \Sigma_{x_k}X_k, \zeta_{x_k} \right)$.
Since each $T_{x_k}X_k$ is defined as
$C_0 \left( \Sigma_{x_k}X_k \right)$,
the ultralimit $\left( \ulim T_{x_k}X_k, 0_{\omega} \right)$
is isometric to $\left( C_0 \left( \ulim \Sigma_{x_k}X_k \right), 0 \right)$.
From our present assumption and Lemma \ref{lem: omdsusp},
we see that
$\ulim \Sigma_{x_k}X_k$ admits an $m$-fold suspender 
$\left( \xi_{1,\omega}, \dots, \xi_{m,\omega} \right)$.

Take the ultralimit $\ulim D_{x_k}\varphi_k$
of the sequence of the maps $D_{x_k}\varphi_k$
defined on $\ulim T_{x_k}X_k$.
Then $\ulim D_{x_k}\varphi_k$ coincides with 
$\left( b_{\gamma_1}, \dots, b_{\gamma_m} \right)$, 
where $b_{\gamma_i}$ is the Busemann function on 
$C_0 \left( \ulim \Sigma_{x_k}X_k \right)$ 
along the ray $\gamma_i$ emanating from $0$ toward $\xi_{i,\omega}$.
By Lemma \ref{lem: busemann},
the map $\ulim D_{x_k}\varphi_k$ is a homogenous submetry.
Take $u \in \Sph^{m-1}$ in 
$\ulim \R^m$ with $u_k \to u$,
where
$\ulim (\R^m,0)$ is the ultralimit of $(\R^m,0)$.
Since the map $\ulim D_{x_k}\varphi_k$ is $1$-open and homogeneous,
we find $\zeta \in C_0 \left( \ulim \Sigma_{x_k}X_k \right)$
with $\left\vert \zeta \right\vert = 1$
such that
$\left( \ulim D_{x_k}\varphi_k \right)(\zeta) = u$.
Choose a sequence $(\zeta_k)$ with $\zeta_k \in \Sigma_{x_k}X_k$ 
contained in $T_{x_k}X_k$
satisfying $\zeta_k \to \zeta$.
Then
$\left( D_{x_k}\varphi_k \right) \left( \zeta_k \right) \to u$.
Hence we have 
$\ulim \Vert \left( D_{x_k}\varphi_k \right) \left( \zeta_k \right) - u_k \Vert = 0$,
and hence for all $\omega$-large $k$
we have 
$\Vert \left( D_{x_k}\varphi_k \right)(\zeta_k) - u_k \Vert < \epsilon/2c$.
Since each $D_{x_k}\varphi_k$ is $c$-open,
for the $u_k \in \Sph^{m-1}$ we find $\eta_k \in T_{x_k}X_k$
with $d_{T_{x_k}X_k} \left( \eta_k, \zeta_k \right) < \epsilon/2$
such that
$\left( D_{x_k}\varphi_k \right) \left( \eta_k \right) = u_k$,
where $d_{T_{x_k}X_k}$ is the metric on $T_{x_k}X_k$.
Therefore we have
$\left\vert \eta_k \right\vert < 1 + \epsilon/2$.
This contradicts \eqref{eqn: str1eopen1}.
\end{proof}

Now we prove Proposition \ref{prop: almsubm}

\begin{proof}[Proof of Proposition \ref{prop: almsubm}]
Take $\epsilon \in (0,1)$ and $m \in \N$.
Let $\delta \in (0,1)$ be sufficiently small.
By Lemma \ref{lem: str1elip},
$\varphi$ is $(1+\epsilon)$-Lipschitz.
As shown in Lemma \ref{lem: str1eopen},
for every $u \in \Sph^{m-1}$,
and for every $x \in X$,
we find $\eta_x \in T_xX$ with 
$\left( D_x\varphi \right) \left( \eta_x \right) = u$
and $\left\vert \eta_x \right\vert < 1+\epsilon$.
Applying the Lytchak open map theorem 
\cite[Theorem 1.2]{lytchak2},
we conclude that $\varphi$ is $(1+\epsilon)$-open.
Thus we prove the proposition.
\end{proof}

\section{Proofs of the main theorems}

\subsection{Proof of Theorem \ref{thm: agrcat0}}

First we conclude the following:

\begin{prop}\label{prop: fullbstr}
For every $\epsilon \in (0,1)$,
and for every $n \in \N$,
there exists 
$\delta \in (0,1)$ satisfying the following:
Let $X$ be a geodesically complete,
complete $\CAT(0)$ space of $\dim_{\mathrm{G}}X \le n$.
If $\varphi \colon X \to \R^n$
is a Busemann $(n,\delta)$-strainer map on $X$,
then it is a $(1+\epsilon)$-bi-Lipschitz homeomorphism.
\end{prop}

\begin{proof}
Let $\delta \in (0,1)$ be sufficiently small.
Let $X$ be a complete $\CAT(0)$ space of $\dim_{\mathrm{G}}X \le n$.
Let $\varphi \colon X \to \R^n$
be a Busemann $(n,\delta)$-strainer map on $X$
with $\varphi = \left( b_{\gamma_1}, \dots, b_{\gamma_n} \right)$.
For each $i \in \{ 1, \dots, m \}$,
we put $\xi_i := \gamma_i(\infty)$.

We prove the injectivity of $\varphi$.
Suppose that
we find distinct points $x, y$ in $X$
with $\varphi(x) = \varphi(y)$.
By Lemma \ref{lem: sphstrinf},
the penetrable $\CAT(1)$ space $\Sigma_xX$
admits an $(n,\delta)$-suspender 
$\left( (\xi_1)_x', \dots, (\xi_n)_x' \right)$.
Since we have $\varphi(x) = \varphi(y)$,
by Lemma \ref{lem: almflattri} we have
$\left\vert \angle_{x}(\xi_i,y) - \pi/2 \right\vert < 2\delta$
for all $i \in \{ 1, \dots, m \}$.
This contradicts Lemma \ref{lem: fullsusp}.
Hence $\varphi$ is injective.

From Proposition \ref{prop: almsubm}
it follows that $\varphi$ is $(1+\epsilon)$-Lipschitz
and $(1+\epsilon)$-open.
This together with the injectivity of $\varphi$
implies that $\varphi$ is 
a $(1+\epsilon)$-bi-Lipschitz homeomorphism.
\end{proof}

From Proposition \ref{prop: fullbstr}
we derive the following:

\begin{thm}
\label{thm: agrcat0g}
For every $\epsilon \in (0,1)$,
and for every $n \in \N$,
there exists $\delta \in (0,1)$ such that
if a geodesically complete $\CAT(0)$ space $X$
of $\dim_{\mathrm{G}} X \le n$
satisfies
$d_{\GH} \left( \partial_{\mathrm{T}}X, \Sph^{n-1} \right) < \delta$,
then $X$ is $(1+\epsilon)$-bi-Lipschitz homeomorphic to $\R^n$.
\end{thm}

\begin{proof}
For $\epsilon \in (0,\infty)$,
and for $n \in \N$,
let $\delta \in (0,1)$ be sufficiently small.
Let $X$ a geodesically complete $\CAT(0)$ space $X$
of $\dim_{\mathrm{G}}X = n$
satisfying
$d_{\GH} \left( \partial_{\mathrm{T}}X, \Sph^{n-1} \right) < \delta$.

In this case,
since we have 
$d_{\GH} \left( \partial_{\mathrm{T}}X, \Sph^{n-1} \right) < \delta$,
there exists an $(n,10\delta)$-strainer at infinity
$\left( \xi_1, \dots, \xi_n \right)$.
For each $i \in \{ 1, \dots, n \}$,
take a ray $\gamma_i \colon [0,\infty) \to X$
with $\xi_i = \gamma_i(\infty)$.
Then the map $\varphi \colon X \to \R^n$
defined by $\varphi = \left( b_{\gamma_1}, \dots, b_{\gamma_n} \right)$
is a Busemann $(n,10\delta)$-strainer map.
From Proposition \ref{prop: fullbstr}
we conclude that 
$\varphi$ is a $(1+\epsilon)$-bi-Lipschitz homeomorphism.
This completes the proof of Theorem \ref{thm: agrcat0}.
\end{proof}

\begin{proof}[Proof of Theorem \ref{thm: agrcat0}]
Theorem \ref{thm: agrcat0g}
and Theorem \ref{thm: dimgdimt} of Kleiner \cite{kleiner}
on geometric dimensions
imply Theorem \ref{thm: agrcat0}.
\end{proof}

\subsection{Proof of Theorem \ref{thm: grcat1}}

We denote by $\vartheta_n \colon (0,\infty) \to (0,\infty)$ 
a positive function
depending only on $n$
satisfying $\lim_{\delta \to 0} \vartheta_n(\delta) = 0$.

Based on the same idea in \cite[Theorem 9.5]{burago-gromov-perelman}
of Burago--Gromov--Perelman
for $\CBB$ spaces
with curvature bounded below,
we prove the following
for geoesically complete $\CAT(1)$ spaces:

\begin{thm}
\label{thm: grcat1g}
For every $\epsilon \in (0,1)$,
and for every $n \in \N$,
there exists $\delta \in (0,1)$ such that
if a geodesically complete $\CAT(1)$ space $Z$
of $\dim_{\mathrm{G}} Z \le n-1$ satisfies
$d_{\GH} \left( Z, \Sph^{n-1} \right) < \delta$,
then $Z$ is $(1+\epsilon)$-bi-Lipschitz homeomorphic to $\Sph^{n-1}$.
\end{thm}

\begin{proof}
For $n \in \N$,
let $\delta \in (0,1)$ be small enough.
Let $Z$ be a geodesically complete $\CAT(1)$ space with metric $d_Z$
of $\dim_{\mathrm{G}}Z \le n-1$
satisfying $d_{\GH} \left( Z, \Sph^{n-1} \right) < \delta$.
Then the Euclidean cone $C_0(Z)$ over $Z$
is a geodesically complete $\CAT(0)$ space
satisfying $\dim_{\mathrm{G}}C_0(Z) \le n$.
Since $Z$ is isometric to $\partial_{\mathrm{T}}C_0(Z)$,
we have 
$d_{\GH} \left( \partial_{\mathrm{T}}C_0(Z), \Sph^{n-1} \right) < \delta$.
In this case,
there exists an $(n,10\delta)$-strainer at infinity
$\left( \xi_1, \dots, \xi_n \right)$.
For each $i \in \{ 1, \dots, n \}$,
take a ray $\gamma_i \colon [0,\infty) \to C_0(Z)$
with $\xi_i = \gamma_i(\infty)$.
Then the map $\varphi_0 \colon C_0(Z) \to \R^n$
defined by 
$\varphi_0 = \left( b_{\gamma_1}, \dots, b_{\gamma_n} \right)$
is a Busemann $(n,10\delta)$-strainer map.
Due to Proposition \ref{prop: fullbstr},
the map
$\varphi_0$ is a 
$\left( 1+\vartheta_n(\delta) \right)$-bi-Lipschitz homeomorphism.

Identify the unit metric sphere centered at the vertex in $C_0(Z)$ with 
$Z$,
and do the unit metric sphere centered at the origin in $\R^n$ with 
$\Sph^{n-1}$.
Define a bijective map $\varphi \colon Z \to \Sph^{n-1}$
by
\[
\varphi(z) := \frac{1}{\Vert \varphi_0(z) \Vert}.
\]
Take distinct points $z_1, z_2$ in $Z$.
If $z_1$ and $z_2$ satisfy $d_Z(z_1,z_2) \ge \pi/100$,
then we see 
\[
\left\vert
d_{\Sph^{n-1}} \left( \varphi(z_1), \varphi(z_2) \right) - 
d_Z \left( z_1, z_2 \right) 
\right\vert
< \left( 1+ \vartheta_n(\delta) \right) d_Z \left( z_1, z_2 \right),
\]
where $d_{\Sph^{n-1}}$ is the metric on $\Sph^{n-1}$.

Assume that 
$z_1$ and $z_2$ satisfy $d_Z(z_1,z_2) < \pi/100$.
Set $l := d_Z(z_1,z_2)$.
Let $\gamma \colon [0,l] \to Z$ be the geodesic
in $Z$ from $z_1$ to $z_2$.
Take $t \in [0,l)$.
By Lemma \ref{lem: sphstrinf},
the $n$-tuple 
$\left( (\xi_1)_{\gamma(t)}', \dots, (\xi_n)_{\gamma(t)}' \right)$
is an $(n,10\delta)$-strainer in $\Sigma_{\gamma(t)}C_0(Z)$.
Let $\xi_{\gamma(t)}$ be the direction $\left (D_t\gamma \right)(1)$
in $\Sigma_{\gamma(t)}C_0(Z)$.
Since $\varphi_0$ has Busemann function coordinates,
we have
\begin{align*}
\left( D_{\gamma(t)} \varphi_0 \right) \left( \xi_{\gamma(t)} \right)
&= - \left( \cos \angle_{\gamma(t)}
\left( \xi_1, \xi_{\gamma(t)} \right), \dots,
\cos \angle_{\gamma(t)} \left( \xi_n, \xi_{\gamma(t)} \right) \right), \\
\left( D_{\gamma(t)} \varphi_0 \right) \left( 0_{\gamma(t)}' \right)
&= - \left( \cos \angle_{\gamma(t)} \left( \xi_1, 0_{\gamma(t)}' \right), 
\dots,
\cos \angle_{\gamma(t)} \left( \xi_n, 0_{\gamma(t)}' \right) \right).
\end{align*}
Since $\varphi_0$ is a 
$\left( 1+\vartheta_n(\delta) \right)$-bi-Lipschitz homeomorphism,
we have
\[
\left\vert 
\left\Vert
\left( D_{\gamma(t)} \varphi_0 \right) \left( \xi_{\gamma(t)} \right) 
\right\Vert - 1
\right\vert < \vartheta_n(\delta),
\quad
\left\vert 
\left\Vert
\left( D_{\gamma(t)} \varphi_0 \right) \left( 0_{\gamma(t)}' \right) 
\right\Vert - 1
\right\vert < \vartheta_n(\delta).
\]
Applying Lemma \ref{lem: aperp} to $\Sigma_{\gamma(t)}C_0(Z)$,
we see
\[
\left\vert 
\angle_{\gamma(t)} \left(
\left( D_{\gamma(t)} \varphi_0 \right) \left( \xi_{\gamma(t)} \right),
\left( D_{\gamma(t)} \varphi_0 \right) \left( 0_{\gamma(t)}' \right) \right)
- \frac{\pi}{2} 
\right\vert < \vartheta_n(\delta);
\]
in other words,
the tangent vector 
$\left( D_{\gamma(t)} \varphi_0 \right) \left( \xi_{\gamma(t)} \right)$
of $\varphi_0 \circ \gamma$ at $t$
and the vector from $(\varphi_0 \circ \gamma)(t)$ to the origin $0$ 
are almost perpendicular.
Hence we see
\[
\left\vert
\frac{
\left\Vert \left( D_{\gamma(t)} \varphi \right) 
\left( \xi_{\gamma(t)} \right) \right\Vert}
{\left\Vert \left( D_{\gamma(t)} \varphi_0 \right) 
\left( \xi_{\gamma(t)} \right) \right\Vert}
-1
\right\vert < \vartheta_n(\delta).
\]
Therefore we obtain
\begin{multline*}
d_{\Sph^{n-1}} \left( \varphi(z_1), \varphi(z_2) \right)
\le \int_0^l \left\Vert \left( D_{\gamma(t)} \varphi \right) 
\left( \xi_{\gamma(t)} \right) \right\Vert dt \\
< \left( 1+\vartheta_n(\delta) \right) 
\int_0^l \left\Vert \left( D_{\gamma(t)} \varphi_0 \right) 
\left( \xi_{\gamma(t)} \right) \right\Vert dt
< \left( 1+\vartheta_n(\delta) \right) d_Z(z_1,z_2).
\end{multline*}

Set $\bar{l} := d_{\Sph^{n-1}} \left( \varphi(z_1), \varphi(z_2) \right)$.
Let $\bar{\gamma} \colon [0,\bar{l}] \to \Sph^{n-1}$ be 
the geodesic in $\Sph^{n-1}$
from $\varphi(z_1)$ to $\varphi(z_2)$, and
$\tilde{\gamma} \colon [0,l] \to \Sph^{n-1}$
the constant-speed reparametrization 
of $\bar{\gamma}$ defined by 
$\tilde{\gamma}(t) := \bar{\gamma} \left( (\bar{l}/l)t \right)$.
Since the tangent vector 
$\left( D_{\gamma(t)} \varphi_0 \right) \left( \xi_{\gamma(t)} \right)$
and the vector from $(\varphi_0 \circ \gamma)(t)$ to the origin $0$
are almost perpendicular,
we see
\[
\left\vert
\frac
{\left\Vert \left( D_{\gamma(t)} \varphi_0 \right) 
\left( \xi_{\gamma(t)} \right) \right\Vert}
{\left\Vert \left( D_t\tilde{\gamma} \right) (1) \right\Vert}
-1
\right\vert < \vartheta_n(\delta).
\]
Therefore we obtain
\begin{multline*}
d_{\Sph^{n-1}} \left( \varphi(z_1), \varphi(z_2) \right)
= \int_0^l \left\Vert \left( D_t\tilde{\gamma} \right) (1) \right\Vert dt \\
> \left( 1-\vartheta_n(\delta) \right) 
\int_0^l \left\Vert \left( D_{\gamma(t)} \varphi_0 \right) 
\left( \xi_{\gamma(t)} \right) \right\Vert dt
= \left( 1-\vartheta_n(\delta) \right) d_Z(z_1,z_2).
\end{multline*}
Thus $\varphi$ is a 
$\left( 1+\vartheta_n(\delta) \right)$-bi-Lipschitz homeomorphism.

In this way, 
we have completed the proof of Theorem \ref{thm: grcat1g}.
\end{proof}

\begin{proof}[Proof of Theorem \ref{thm: grcat1}]
Combining Theorem \ref{thm: grcat1g}
and Theorem \ref{thm: dimgdimt} of Kleiner \cite{kleiner}
on geometric dimensions,
we conclude Theorem \ref{thm: grcat1}.
\end{proof}




\begin{thebibliography}{99}

\bibitem{alexander-kapovitch-petrunin-0}
S. Alexander, V. Kapovitch, and A. Petrunin,
\emph{An Invitation to Alexandrov Geometry: $\CAT(0)$ Spaces},
Springer Briefs in Mathematics,
Springer, 2019.

\bibitem{alexander-kapovitch-petrunin}
S. Alexander, V. Kapovitch, and A. Petrunin,
\emph{Alexandrov Geometry: Foundations}, 
Graduate Studies in Mathematics, Volume 236, Amer. Math. Soc.,
2024.

\bibitem{ancel-davis-guilbault}
F. D. Ancel, M. W. Davis, and C. R. Guilbault,
\emph{$\CAT(0)$ reflection manifolds},
Geometric Topology, Part 1 (W. H. Kazez, ed.),
AMS/IP Studies in Advanced Mathematics, Volume 2,
Amer. Math. Soc., 1997,
pp.441--445.

\bibitem{ballmann}
W. Ballmann,
\emph{Lectures on Spaces of Nonpositive Curvature},
DMV Seminar, Band 25, Birkh\"{a}user, 1995.

\bibitem{ballmann-gromov-schroeder}
W. Ballmann, M. Gromov, and V. Schroeder,
\emph{Manifolds of Nonpositive Curvature},
Progress in Mathematics, Volume 61, Birkh\"{a}user, 1985.

\bibitem{bridson-haefliger}
M. R. Bridson and A. Haefliger,
\emph{Metric Spaces of Non-Positive Curvature},
Grundlehren der Mathematischen Wissenshaften,
Volume 319,
Springer-Verlag,
1999.

\bibitem{brown}
M. Brown,
\emph{The monotone union of open cells is an open ball},
Proc. Amer. Math. Soc. {\bf 12} (1961), 812--814.

\bibitem{burago-burago-ivanov}
D. Burago, Yu. D. Burago, and S. Ivanov,
\emph{A Course in Metric Geometry},
Graduate Studies in Mathematics, Volume 33, Amer. Math. Soc.,
2001.

\bibitem{burago-gromov-perelman}
Yu. D. Burago, M. Gromov, and G. Perelman,
\emph{A. D. Alexandrov spaces with curvature bounded below}
[Russian],
Uspekhi Mat. Nauk {\bf 47} (1992), no.~2 (284), 3--51, 222;
translation in 
Russian Math. Surveys {\bf 47} (1992), no.~2, 1--58.

\bibitem{davis-januszkiewicz}
M. W. Davis and T. Januszkiewicz,
\emph{Hyperbolization of polyhedra},
J. Differential Geom. {\bf 34} (1991), 347--388.

\bibitem{fujioka-gu}
T. Fujioka and S. Gu,
\emph{Topological regularity of Busemann spaces of nonpositive curvature},
preprint, 2025:
arXiv:2504.14455.

\bibitem{fujiwara-nagano-shioya}
K. Fujiwara, K. Nagano, and T. Shioya,
\emph{Fixed point sets of parabolic isometries of $\CAT(0)$-spaces},
Comment. Math. Helv. {\bf 81} (2006), 305--335.

\bibitem{gromovq}
M. Gromov,
\emph{Hyperbolic manifolds, groups and actions},
Riemann Surfaces and Related Topics:
Proceedings of the 1978 Stony Brook Conference 
(I. Kra and B. Maskit, eds.),
Annals of Mathematics Studies, Volume 97,
Princeton University Press, 1981,
pp.~183--213.

\bibitem{gromovh}
M. Gromov,
\emph{Hyperbolic groups},
Essays in Group Theory (S. M. Gersten, ed.), 
Mathematical Sciences Research Institute Publication, Volume 8, Springer, 1987,
pp.~75--263.

\bibitem{kleiner}
B. Kleiner,
\emph{The local structure of length spaces with curvature bounded above},
Math. Z. {\bf 231} (1999), 409--456.

\bibitem{leeb}
B. Leeb,
\emph{A characterization of irreducible symmetric spaces and Euclidean buildings
of higher rank by their asymptotic geometry},
Bonner Mathematische Schriften 326,
Universit\"{a}t Bonn, Mathematisches Institut, 2000.

\bibitem{lytchak2}
A. Lytchak,
\emph{Open map theorem in metric spaces},
Algebra i Analiz {\bf 17} (2005), no.~3, 139--159;
St. Petersburg Math. J. {\bf 17} (2006), no.~3, 477--491.

\bibitem{lytchak-nagano1}
A. Lytchak and K. Nagano,
\emph{Geodesically complete spaces with an upper curvature bound},
Geom. Funct. Anal. {\bf 29} (2019), no.~1, 295--342.

\bibitem{lytchak-nagano2}
A. Lytchak and K. Nagano,
\emph{Topological regularity of spaces with an upper curvature bound},
J. Eur. Math. Soc. (JEMS) {\bf 24} (2022), no.~1, 137--165.

\bibitem{lytchak-nagano-stadler}
A. Lytchak, K. Nagano, and S. Stadler,
\emph{$\CAT(0)$ $4$-manifolds are Euclidean},
Geom. Topol. 28 (2024), no.~7, 3285--3308.

\bibitem{nagano4}
K. Nagano,
\emph{Volume pinching theorems for $\CAT(1)$ spaces},
Amer. J. Math. {\bf 144} (2022), no.~1, 267--285.

\bibitem{nagano5}
K. Nagano,
\emph{Asymptotic topological regularity of $\CAT(0)$ spaces}
Ann. Global Anal. Geom. {\bf 61} (2022), no.~2, 427--457.

\bibitem{nagano6}
K. Nagano,
\emph{Wall singularity of spaces with an upper curvature bound}
preprint, 2026:
arXiv:2601.22673.

\bibitem{rolfsen}
D. Rolfsen,
\emph{Strongly convex metrics in cells},
Bull. Amer. Math. Soc. {\bf 78} (1968), 171--175.

\bibitem{stallings}
J. Stallings,
\emph{The piecewise-linear structure of Euclidean spaces},
Proc. Cambridge Phil. Soc. {\bf 58} (1962), 481--488. 

\bibitem{thurston}
P. Thurston,
\emph{$\CAT(0)$ $4$-manifolds possessing a single tame point are Euclidean},
J. Geom. Anal. {\bf 6} (1996), no.~3, 475--494.

\end{thebibliography}
\end{document}